\documentclass[a4paper, 12pt]{article}

\usepackage[utf8]{inputenc}
\usepackage[T1]{fontenc}
\usepackage[a4paper, margin=2.5cm]{geometry}
\usepackage{needspace}

\usepackage{libertine} % text font
\usepackage{inconsolata} % texttt font

\usepackage{amsmath, amsthm, amssymb}
\usepackage{graphicx}
\usepackage{enumerate}
\usepackage{authblk}
\usepackage[textsize=small, textwidth=2cm, color=yellow]{todonotes}
\usepackage{thmtools}
\usepackage{thm-restate}
\usepackage{todonotes}
\usepackage[colorlinks=true, citecolor=red]{hyperref}
\usepackage{float}

\renewenvironment{abstract}
{\small\vspace{-1em}
\begin{center}
\bfseries\abstractname\vspace{-.5em}\vspace{0pt}
\end{center}
\list{}{
\setlength{\leftmargin}{0.6in}%
\setlength{\rightmargin}{\leftmargin}}%
\item\relax}
{\endlist}

\declaretheorem[name=Theorem, numberwithin=section]{theorem}

\declaretheorem[name=Conjecture, sibling=theorem]{conjecture}

\declaretheorem[name=Claim, sibling=theorem]{claim}

\declaretheorem[name=Question, style=remark, sibling=theorem]{question}
\def\cqedsymbol{\ifmmode$\lrcorner$\else{\unskip\nobreak\hfil
\penalty50\hskip1em\null\nobreak\hfil$\lrcorner$
\parfillskip=0pt\finalhyphendemerits=0\endgraf}\fi} 
\newcommand{\cqed}{\renewcommand{\qed}{\cqedsymbol}}

\let\leq\leqslant
\let\geq\geqslant

\title{On a recolouring version of Hadwiger's conjecture\footnote{The authors are supported by ANR project GrR (\textsc{ANR-18-CE40-0032}).}}

\author[1]{Marthe Bonamy}
\author[2]{Marc Heinrich}
\author[3]{Clément Legrand-Duchesne}
\author[1]{Jonathan Narboni}
\affil[1]{CNRS, LaBRI, Université de Bordeaux, Bordeaux, France.}
\affil[2]{University of Leeds, United Kingdom.}
\affil[3]{Univ Rennes, F-35000 Rennes, France.}
\date{\today}

\begin{document}

\maketitle
\begin{center}
\begin{abstract}
We prove that for any $\varepsilon>0$, for any large enough $t$, there is a graph that admits no $K_t$-minor but admits a $(\frac32-\varepsilon)t$-colouring that is ``frozen'' with respect to Kempe changes, i.e. any two colour classes induce a connected component. This disproves three conjectures of Las Vergnas and Meyniel from 1981.
%It's very interesting and important.
\end{abstract}
\end{center}

\section{Introduction}\label{sec:intro}

In an attempt to prove the Four Colour Theorem in 1879, Kempe~\cite{kempe1879geographical} introduced an elementary operation on the colourings\footnote{Throughout this paper, all colourings are proper, i.e. no two vertices with the same colour are adjacent.} of a graph that became known as a Kempe change. Given a $k$-colouring $\alpha$ of a graph $G$, a \emph{Kempe chain} is a maximal bichromatic component\footnote{If a vertex of $G$ is coloured $1$ and has no neighbour coloured $2$ in $\alpha$, then it forms a Kempe chain of size $1$.}. A \emph{Kempe change} in $\alpha$ corresponds to swapping the two colours of a Kempe chain so as to obtain another $k$-colouring. Two $k$-colourings are \emph{Kempe equivalent} if one can be obtained from the other through a series of Kempe changes. %Note that the resulting $k$-colouring is indeed proper.

The study of Kempe changes has a vast history, see e.g.~\cite{mohar2006kempe} for a comprehensive overview or~\cite{bonamy2019conjecture} for a recent result on general graphs. We refer the curious reader to the relevant chapter of a 2013 survey by Cereceda~\cite{van2013complexity}. Kempe equivalence falls within the wider setting of combinatorial reconfiguration, to which~\cite{van2013complexity} is also an excellent introduction. Perhaps surprisingly, Kempe equivalence has direct applications in approximate counting and applications in statistical physics (see e.g.\ \cite{sokal2000personal,mohar2009new} for nice overviews). Closer to graph theory, Kempe equivalence can be studied with a goal of obtaining a random colouring by applying random walks and rapidly mixing Markov chains, see e.g.~\cite{vigoda}.

Kempe changes were introduced as a mere tool, and are decisive in the proof of Vizing's edge colouring theorem~\cite{vizing1964estimate}. However, the equivalence class they define on the set of $k$-colourings is itself highly interesting. In which cases is there a single equivalence class? In which cases does every equivalence class contain a colouring that uses the minimum number of colours? Vizing conjectured in 1965~\cite{vizing1968some} that the second scenario should be true in every line graph, no matter the choice of $k$. Despite partial results~\cite{asratian2009note,casselgren}, this conjecture remains wildly open. %Should be widely, but wildly is good too?

In the setting of planar graphs, Meyniel proved in 1977~\cite{meyniel1978colorations} that all $5$-colourings form a unique Kempe equivalence class. The result was then extended to all $K_5$-minor-free graphs in 1979 by Las Vergnas and Meyniel~\cite{las1981kempe}. They conjectured the following, which can be seen as a reconfiguration counterpoint to Hadwiger's conjecture, though it neither implies it nor is implied by it.

\begin{conjecture}[Conjecture A in~\cite{las1981kempe}]\label{conj:A}
For every $t$, all the $t$-colourings of a graph with no $K_t$-minor form a single equivalence class.
\end{conjecture}

They also proposed a related conjecture that is weaker assuming Hadwiger's conjecture holds.

\begin{conjecture}[Conjecture A' in~\cite{las1981kempe}]\label{conj:A'}
For every $t$ and every graph with no $K_t$-minor, every equivalence class of $t$-colourings contains some $(t-1)$-colouring.
\end{conjecture}

Here, we disprove both Conjectures~\ref{conj:A} and~\ref{conj:A'}, as follows.

\begin{theorem}\label{th:mainweak}
For every $\varepsilon>0$ and for any large enough $t$, there is a graph with no $K_t$-minor, whose $(\frac32-\varepsilon)t$-colourings are not all Kempe equivalent. 
\end{theorem}

In fact, we prove that for every $\varepsilon>0$ and for any large enough $t$, there is a graph that does not admit a $K_t$-minor but admits a $(\frac32-\varepsilon)t$-colouring that is \emph{frozen}: Any pair of colours induces a connected component, so the resulting colour partition is an invariant under Kempe changes. To obtain Theorem~\ref{th:mainweak}, we then argue that the graph admits a colouring with a different colour partition. 

The notion of frozen $k$-colouring is related to that of \emph{quasi-$K_p$-minor}, introduced in \cite{las1981kempe}. A graph $G$ admits a $K_p$-minor if it admits $p$ non-empty, pairwise disjoint and connected \emph{bags} $B_1,\ldots,B_p \subset V(G)$ such that for any $i\neq j$, there is an edge between some vertex in $B_i$ and some vertex in $B_j$. For the notion of quasi-$K_p$-minor, we drop the restriction that each $B_i$ should induce a connected subgraph of $G$, and replace it with the condition that for any $i \neq j$, the set $B_i \cup B_j$ induces a connected subgraph of $G$. If the graph $G$ admits a frozen $p$-colouring, then it trivially admits a quasi-$K_p$-minor\footnote{One bag for each colour class.}, while the converse may not be true. 

If all $p$-colourings of a graph form a single equivalence class, then either there is no frozen $p$-colouring or there is a unique $p$-colouring of the graph up to colour permutation. The latter situation in a graph with no $K_p$-minor would disprove Hadwiger's conjecture, so Las Vergnas and Meyniel conjectured that there is no frozen $p$-colouring in that case. Namely, they conjectured the following. 

%Since the latter situation in a graph with no $K_p$-minor would disprove Hadwiger's conjecture, Las Vergnas and Meyniel raised the following conjecture.

\begin{conjecture}[Conjecture C in~\cite{las1981kempe}]\label{conj:C}
For any $t$, any graph that admits a quasi-$K_t$-minor admits a $K_t$-minor.
\end{conjecture}

%Informally, Conjecture~\ref{conj:C} states that a graph with no $K_t$-minor admits no frozen $t$-colouring.  
Conjecture~\ref{conj:C} was proved for increasing values of $t$, and is now
known to hold for
${t\leq10}$~\cite{jorgensen1994contractions,song2006extremal,kriesell21uniquely}. As discussed above, we strongly disprove Conjecture~\ref{conj:C} for large $t$. It is unclear how large $t$ needs to be for a counter-example.

\begin{theorem}\label{th:mainfull}
For every $\varepsilon>0$ and for any large enough $t$, there is a graph that admits a quasi-$K_t$-minor but does not admit a $K_{\left(\frac23+\varepsilon\right)t}$-minor. 
\end{theorem}

Trivially, every graph that admits a quasi-$K_{2t}$-minor admits a $K_t$-minor. We leave the following two open questions, noting that $\frac23 \geq c\geq \frac12$ and $c'\geq \frac32$.

\begin{question}\label{qu:quasi}
What is the infimum $c$ such that for any large enough $t$, there is a graph that admits a quasi-$K_t$-minor but no $K_{ct}$-minor?
\end{question}

%We know $\frac23 \geq c\geq \frac12$, and refrain from making any guess as to the right value.

\begin{question}\label{qu:Hadwiger}
Is there a constant $c'$ such that for every $t$, all the $c'\cdot t$-colourings of a graph with no $K_t$-minor form a single equivalence class?
\end{question}

In the 1980's,~\cite{kostochka1982minimum,kostochka1984lower}
and~\cite{thomason1984extremal} proved independently that a graph with no
$K_t$-minor has degeneracy $O(t\sqrt{\log t})$, with current best hidden factor in~\cite{kelly2020local}. Since all the $k$-colorings of $d$-degenerate graphs are
equivalent for $k > d$~\cite{las1981kempe}, this gives the best upper bound
known so far for Question~\ref{qu:Hadwiger}.

%The study of Kempe changes has a vast history, see e.g.~\cite{mohar2006kempe} for a comprehensive overview or~\cite{bonamy2019conjecture} for a recent result on general graphs. We refer the curious reader to the relevant chapter of a 2013 survey by Cereceda~\cite{van2013complexity}. Kempe-equivalence falls within the wider setting of combinatorial reconfiguration, which~\cite{van2013complexity} is also an excellent introduction to. Perhaps surprisingly, Kempe-equivalence has direct applications in approximate counting and applications in statistical physics (see e.g.\ \cite{sokal2000personal,mohar2009new} for nice overviews). Closer to graph theory, Kempe-equivalence can be studied with a goal of obtaining a random colouring by applying random walks and rapidly mixing Markov chains, see e.g.~\cite{vigoda}.

\section{Construction}

Let $n \in \mathbb{N}$ and let $\varepsilon>0$. We build a random graph $G_n$ on vertex set $\{a_1, b_1, a_2, b_2\ldots,a_n,b_n\}$: for every $i \neq j$ independently, we select one pair uniformly at random among $\{(a_i,a_j),(a_i,b_j),(b_i,a_j),(b_i,b_j)\}$ and add the three other pairs as edges to the graph $G_n$.

%Let $G_n$ be the graph obtained from $K_{2n}$ by removing a matching $(a_i,b_i)$ and for every $i \neq j$, deleting a random edge out of the four $\{(a_i,a_j),(a_i,b_j),(b_i,a_j),(b_i,b_j)\}$.

Note that the sets $\{a_i,b_i\}_{1 \leq i \leq n}$ form a quasi-$K_n$-minor, as for every $i\neq j$, the set $\{a_i,b_i,a_j,b_j\}$ induces a path on four vertices in $G_n$, hence is connected.

Our goal is to argue that if $n$ is sufficiently large then with high probability the graph $G_n$ does not admit any $K_{\left(\frac23 +\varepsilon\right)n}$-minor. This will yield Theorem~\ref{th:mainfull}. To additionally obtain Theorem~\ref{th:mainweak}, we need to argue that with high probability, $G_n$ admits an $n$-colouring with a different colour partition than the natural one, where the colour classes are of the form $\{a_i,b_i\}$. Informally, we can observe that each of  $\{a_1,\ldots,a_n\}$ and $\{b_1,\ldots,b_n\}$ induces a graph behaving like a graph in $\mathcal{G}_{n,\frac34}$ (i.e. each edge exists with probability $\frac34$) though the two processes are not independent. This argument indicates that $\chi(G_n)=O(\frac{n}{\log n})$, but we prefer a simpler, more pedestrian approach.

Assume that for some $i,j,k,\ell$, none of the edges $a_ib_j$, $a_jb_k$, $a_kb_\ell$ and $a_\ell b_i$ exist. Then the graph $G_n$ admits an $n$-colouring $\alpha$ where $\alpha(a_p)=\alpha(b_p)=p$ for every $p \not\in \{i,j,k,\ell\}$ and $\alpha(a_i)=\alpha(b_j)=i$, $\alpha(a_j)=\alpha(b_k)=j$, $\alpha(a_k)=\alpha(b_\ell)=k$ and $\alpha(a_\ell)=\alpha(b_i)=\ell$ (see Figure~\ref{fig:clement}). Since every quadruple $(i,j,k,\ell)$ has a positive and constant probability of satisfying this property, $G_n$ contains such a quadruple with very high probability when $n$ is large.

\begin{figure}[!h]
\center
%\begin{tikzpicture}
%    \tikzstyle{whitenode}=[draw,circle,fill=white,minimum size=9pt,inner sep=0pt]
%    \tikzstyle{blacknode}=[draw,circle,fill=black,minimum size=7pt,inner sep=0pt]
%    \draw (0,0) node[blacknode] (a) {}
%-- ++(0:1cm) node[blacknode] (b) {}
%-- ++(72:1cm) node[blacknode] (c) {}
%-- ++(2*72:1cm) node[blacknode] (d) {}
%-- ++(3*72:1cm) node[blacknode] (e) {};
%\draw (a) -- (e);
%\end{tikzpicture}
\includegraphics[width=.5\textwidth]{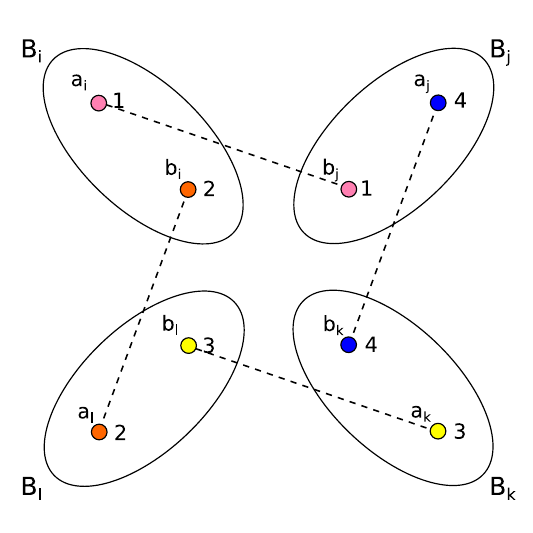}
\caption{\label{fig:clement}A different $n$-colouring given an appropriate quadruple.}
\end{figure}

%Note that $G_n$ admits a quasi $K_n$-minor, with each bag being $(a_i,b_i)$ for some $i$.

We are now ready to prove that the probability that $G_n$ admits a $K_{\left(\frac23+\varepsilon\right)n}$-minor tends to $0$ as $n$ grows to infinity. We consider three types of $K_p$-minors in $G$, depending on the size of the bags involved. If every bag is of size $1$, we say that it is a \emph{simple} $K_p$-minor -- in fact, it is a subgraph. If every bag is of size $2$, we say it is a \emph{double} $K_p$-minor. If every bag is of size at least $3$, we say it is a \emph{triple} $K_p$-minor. We prove three claims, as follows.

\begin{claim}\label{lem:simple}
For any $\delta >0$, $\mathbb{P}( G_n \textrm{ contains a simple }K_{\delta n}\textrm{-minor}) \rightarrow 0$ as $n \rightarrow \infty$.
\end{claim}
\begin{claim}\label{lem:double}
For any $\delta >0$, $\mathbb{P}( G_n \textrm{ contains a double }K_{\delta n}\textrm{-minor}) \rightarrow 0$ as $n \rightarrow \infty$.
\end{claim}
\begin{claim}\label{lem:triple}
$G_n$ does not contain a triple $K_{\frac23n+1}$-minor.
\end{claim}

Claims~\ref{lem:simple},~\ref{lem:double} and~\ref{lem:triple} are proved in Sections~\ref{sect:simple},~\ref{sect:double} and~\ref{sect:triple}, respectively. If a graph admits a $K_p$-minor, then in particular it admits a simple $K_a$-minor, a double $K_b$-minor and a triple $K_c$-minor such that $a+b+c \geq p$. Setting $\delta$ appropriately and combining Claims~\ref{lem:simple},~\ref{lem:double} and~\ref{lem:triple}, we derive the desired conclusion.

\subsection{No large simple minor}\label{sect:simple}

\begin{proof}[Proof of Claim~\ref{lem:simple}]
Let $S$ be a subset of $k$ vertices of $G_n$. The probability that $S$ induces a clique in $G_n$ is at most 
$\left(\frac34\right)^{k \choose 2}$. 
Indeed, if $\{a_i,b_i\}\subseteq S$ for some $i$, then the probability is $0$. Otherwise, $|S  \cap \{a_i,b_i\}| \leq 1$ for every $i$, so we have $G[S] \in \mathcal{G}_{k,\frac34}$, i.e. edges exist independently with probability $\frac34$. Therefore, the probability that $S$ induces a clique is $\left(\frac34\right)^{k \choose 2}$. % (it is $0$ if and only if it contains two vertices of the same bag).

By union-bound, the probability that some subset on $k$ vertices induces a clique is at most ${2n \choose k}\cdot \left(\frac34\right)^{k \choose 2}$. For any $\delta>0$, we note that ${2n \choose \delta n}\leq 2^{2n}$. Therefore, the probability that $G_n$ contains a simple $K_{\delta n}$-minor is at most $2^{2n}\cdot \left(\frac34\right)^{{\delta n} \choose 2}$, which tends to $0$ as $n$ grows to infinity.\cqed
%we note that $\mathbb{P}(\omega(G_n) < \delta n) \geq 1 - 2^{2n}\cdot \left(\frac34\right)^{{\delta n} \choose 2}$, which tends to $1$ as $n$ grows to infinity.\cqed
\end{proof}

\subsection{No large double minor}\label{sect:double}

\begin{proof}[Proof of Claim~\ref{lem:double}]

Let $S'$ be a subset of $k$ pairwise disjoint pairs of vertices in $G_n$ such that for every $i$, at most one of $\{a_i,b_i\}$ is involved in $S'$.

We consider the probability that $G_n/_{S'}$ induces a clique, where $G_n/_{S'}$ is defined as the graph obtained from $G_n$ by considering only vertices involved in some pair of $S'$ and identifying the vertices in each pair.

We consider two distinct pairs $(x_1,y_1),(x_2,y_2)$ of $S'$. Without loss of generality, $\{x_1,x_2,y_1,y_2\}=\{a_i,a_j,a_k,a_\ell\}$ for some $i,j,k,\ell$. The probability that there is an edge between $\{x_1,y_1\}$ and $\{x_2,y_2\}$ is $1-\left(\frac14\right)^4$. In other words, $\mathbb{P}(E((x_1,y_1),(x_2,y_2))= \emptyset)=\left(\frac14\right)^4$ and since at most one of $\{a_i, b_i\}$ is involved in $S'$ for all $i$, all such events are mutually independent. Therefore, the probability that $S'$ yields a quasi-$K_{|S'|}$-minor is $\left(1-\left(\frac14\right)^4\right)^{|S'| \choose 2}$.

For any $\delta' >0$, the number of candidates for $S'$ is at most ${2n \choose {2\delta' n}}$ (the number of choices for a ground set of $2\delta' n$ vertices) times $(2 \delta' n)!$ (a rough upper bound on the number of ways to pair them). Note that ${2n \choose {2\delta' n}}\cdot (2 \delta' n)!  \leq (2n)^{2 \delta' n}$. 
We derive that the probability that there is a set $S'$ of size $\delta' n$ such that $G_n/_{S'}=K_{|S'|}$ is at most $(2n)^{2 \delta' n} \cdot \left(1-\left(\frac14\right)^4\right)^{{\delta' n} \choose 2}$, which tends to $0$ as $n$ grows large.
%$ (2n)^{2 \delta' n}\left(1-\left(\frac14\right)^4\right)^{\frac12 \cdot \delta'^2 n^2}$, which tends to $0$ as $n$ grows large.

Consider a double $K_k$-minor $S$ of $G_n$. Note that no pair in $S$ is equal to $\{a_i,b_i\}$ (for any $i$), as every bag induces a connected subgraph in $G_n$. We build greedily a maximal subset $S'\subseteq S$ such that $S'$ involves at most one vertex out of every set of type $\{a_i,b_i\}$. Note that $|S'|\geq \frac{|S|}3$. Therefore, by taking $\delta'=\frac{\delta}3$ in the above analysis, we obtain that the probability that there is a set $S$ of $\delta n$ pairs that induces a quasi-$K_{|S|}$-minor tends to $0$ as $n$ grows large.\cqed

%Let $S$ be a subset of $k$ pairwise disjoint pairs of vertices in $G_n$ such that no pair is equal to $\{a_i,b_i\}$ for any $i$. We build greedily a maximal subset $S'\subseteq S$ such that $S'$ involves at most one vertex out of every set of type $\{a_i,b_i\}$. Note that $|S'|\geq \frac{|S|}3$. By taking $\delta'=\frac{\delta}3$ in the above analysis, we obtain that the probability that there is a set $S$ of $\delta n$ pairs that induces a quasi-$K_{|S|}$-minor tends to $0$ as $n$ grows large.\cqed
\end{proof}

\subsection{No large triple minor}\label{sect:triple}

\begin{proof}[Proof of Claim~\ref{lem:triple}]
The graph $G_n$ has $2n$ vertices, and a triple $K_k$-minor involves at least $3k$ vertices. It follows that if $G_n$ contains a triple $K_k$-minor then $k \leq \frac{2n}3$.\cqed
\end{proof}

\section*{Addendum}

Following the submission and advertisement of this paper, we learned that the exact same construction had been independently discovered in~\cite{counterexamples1993Reed} and~\cite{bohme2010hadwiger}. Its authors and the rest of the community were however seemingly unaware of the implications on quasi-minors and Kempe equivalence that we discuss here.

\section*{Acknowledgements}

The authors thank Vincent Delecroix for helpful discussions, as well as Louis Esperet, Zi-Xia Song and
Raphael Steiner for pointing out additional relevant references,
specifically~\cite{counterexamples1993Reed},~\cite{kriesell21uniquely,song2006extremal} and~\cite{bohme2010hadwiger}.

%\section{Conclusion}\label{sec:ccl}

\bibliographystyle{alpha}
\bibliography{biblio}

\begin{thebibliography}{{Mey}78}

\bibitem[AC16]{casselgren}
Armen~S. Asratian and Carl~Johan Casselgren.
\newblock Solution of {V}izing's problem on interchanges for the case of graphs
  with maximum degree 4 and related results.
\newblock {\em Journal of Graph Theory}, 82(4):350--373, 2016.

\bibitem[Asr09]{asratian2009note}
Armen~S. Asratian.
\newblock A note on transformations of edge colorings of bipartite graphs.
\newblock {\em Journal of Combinatorial Theory, Series B}, 99(5):814--818,
  2009.

\bibitem[BBFJ19]{bonamy2019conjecture}
Marthe Bonamy, Nicolas Bousquet, Carl Feghali, and Matthew Johnson.
\newblock On a conjecture of {M}ohar concerning {K}empe equivalence of regular
  graphs.
\newblock {\em Journal of Combinatorial Theory, Series B}, 135:179--199, 2019.

\bibitem[BKT10]{bohme2010hadwiger}
Thomas Böhme, Alexandr Kostochka, and Andrew Thomason.
\newblock Hadwiger numbers and over-dominating colourings.
\newblock {\em Discrete Mathematics}, 310(20):2662--2665, 2010.
\newblock Graph Theory — Dedicated to Carsten Thomassen on his 60th Birthday.

\bibitem[J{\o}r94]{jorgensen1994contractions}
Leif~K. J{\o}rgensen.
\newblock Contractions to {K}8.
\newblock {\em Journal of Graph Theory}, 18(5):431--448, 1994.

\bibitem[Kem79]{kempe1879geographical}
Alfred~B. Kempe.
\newblock On the geographical problem of the four colours.
\newblock {\em American journal of mathematics}, 2(3):193--200, 1879.

\bibitem[Kos82]{kostochka1982minimum}
Alexandr~V Kostochka.
\newblock The minimum {H}adwiger number for graphs with a given mean degree of
  vertices.
\newblock {\em Metody Diskret. Analiz.}, (38):37--58, 1982.

\bibitem[Kos84]{kostochka1984lower}
Alexandr~V. Kostochka.
\newblock Lower bound of the {H}adwiger number of graphs by their average
  degree.
\newblock {\em Combinatorica}, 4(4):307--316, 1984.

\bibitem[KP20]{kelly2020local}
Tom Kelly and Luke Postle.
\newblock A local epsilon version of {R}eed's conjecture.
\newblock {\em Journal of Combinatorial Theory, Series B}, 141:181--222, 2020.

\bibitem[Kri21]{kriesell21uniquely}
Matthias Kriesell.
\newblock A note on uniquely 10-colorable graphs.
\newblock {\em Journal of Graph Theory}, 98(1):24--26, 2021.

\bibitem[LVM81]{las1981kempe}
Michel Las~Vergnas and Henri Meyniel.
\newblock Kempe classes and the {H}adwiger conjecture.
\newblock {\em Journal of Combinatorial Theory, Series B}, 31(1):95--104, 1981.

\bibitem[{Mey}78]{meyniel1978colorations}
Henry {Meyniel}.
\newblock {Les 5-colorations d'un graphe planaire forment une classe de
  commutation unique}.
\newblock {\em {J. Comb. Theory, Ser. B}}, 24:251--257, 1978.

\bibitem[Moh06]{mohar2006kempe}
Bojan Mohar.
\newblock Kempe equivalence of colorings.
\newblock In {\em Graph Theory in Paris}, pages 287--297. Springer, 2006.

\bibitem[MS09]{mohar2009new}
Bojan Mohar and Jes{\'u}s Salas.
\newblock A new {K}empe invariant and the (non)-ergodicity of the
  {W}ang--{S}wendsen--{K}oteck{\`y} algorithm.
\newblock {\em Journal of Physics A: Mathematical and Theoretical},
  42(22):225204, 2009.

\bibitem[Ree93]{counterexamples1993Reed}
Bruce Reed.
\newblock Counterexamples to a conjecture of {Las} {Vergnas} and {Meyniel}.
\newblock In {\em Graph structure theory. Proceedings of the AMS-IMS-SIAM joint
  summer research conference on graph minors held June 22 to July 5, 1991 at
  the University of Washington, Seattle, WA (USA)}, pages 157--159. Providence,
  RI: American Mathematical Society, 1993.

\bibitem[Sok00]{sokal2000personal}
Alan~D. Sokal.
\newblock A personal list of unsolved problems concerning lattice gases and
  antiferromagnetic {P}otts models.
\newblock {\em arXiv preprint cond-mat/0004231}, 2000.

\bibitem[ST06]{song2006extremal}
Zi-Xia Song and Robin Thomas.
\newblock The extremal function for {$K_9$} minors.
\newblock {\em Journal of Combinatorial Theory, Series B}, 96(2):240--252,
  2006.

\bibitem[Tho84]{thomason1984extremal}
Andrew Thomason.
\newblock An extremal function for contractions of graphs.
\newblock In {\em Mathematical Proceedings of the Cambridge Philosophical
  Society}, volume~95, pages 261--265. Cambridge University Press, 1984.

\bibitem[vdH13]{van2013complexity}
Jan van~den Heuvel.
\newblock The complexity of change.
\newblock {\em Surveys in combinatorics}, 409(2013):127--160, 2013.

\bibitem[Vig00]{vigoda}
Eric Vigoda.
\newblock Improved bounds for sampling colorings.
\newblock {\em Journal of Mathematical Physics}, 41(3):1555--1569, 2000.

\bibitem[Viz64]{vizing1964estimate}
Vadim~G. Vizing.
\newblock On an estimate of the chromatic class of a p-graph.
\newblock {\em Discret Analiz}, 3:25--30, 1964.

\bibitem[Viz68]{vizing1968some}
Vadim~G. Vizing.
\newblock Some unsolved problems in graph theory.
\newblock {\em Russian Mathematical Surveys}, 23(6):125, 1968.

\end{thebibliography}

\end{document}